\documentclass[12pt]{article}
\usepackage{amsmath,amsthm,amsfonts,latexsym,amssymb,hyperref}
\usepackage{graphicx}

\newtheorem{definition}{\bf{Definition}}[section]
\newtheorem{lemma}{\bf{Lemma}}[section]
\newtheorem{prop}{\bf{Proposition}}[section]
\newtheorem{theorem}{\bf{Theorem}}[section]
\newtheorem{remark}{\sc{Remark}}[section]

\begin{document}
\title{Dynamic bilateral boundary conditions on interfaces}

\author{Luisa Consiglieri
  \footnote{Independent Research Professor, Portugal. 
\href{http://sites.google.com/site/luisaconsiglieri}{http://sites.google.com/site/luisaconsiglieri}
}}
\maketitle
\begin{abstract}
Two boundary value problems for an elliptic equation  in divergence form
with bounded discontinuous coefficient are studied in a bidomain.
On the interface, generalized dynamic boundary conditions such as of the Wentzell-type and
Signorini-type transmission are considered in a subdifferential form.
Several non-constant coefficients and nonlinearities are the main objective of the present work. 
Generalized solutions are built via time discretization.
\end{abstract}

{\bf Keywords.} Wentzell transmission, Signorini transmission, subdifferential, Rothe method

\bf 2000 MSC \rm{35J87, 49M25, 78A70}%

\section{Introduction}

In the description of real life phenomena,
challenges in science and technology such as diffusion problems with transmission conditions
are being addressed  (cf. for instance \cite{colli} and the references therein).
We refer to \cite{enfr,favi} a general framework which allows  to prove, in
a unified and systematic way, the analyticity of semigroups generated by
operators with generalized Wentzell boundary conditions on function spaces
with bounded trace operators.
The thin obstacle problem (also called the Signorini problem)
  models threshold phenomena
like contact problem, thermostatic device or semi-permeable membranes \cite{br}.
In \cite{amar} the study  relies on the presence of differential operators.
We point out that their method is based on a fixed point argument.
Under continuous or even constant coefficients,
the regularity  was shown for the Laplace-Wentzell problem \cite{engel}
or the thin obstacle problem \cite{chi}.
The question of dynamic boundary conditions can be found in  frictional contact problems
(see \cite{rp} and the references therein).
Their theoretical and numerical achievements are based on the time discretization method being closely related to ours.

With the aim of forcing to make realistic assumptions and then deal with the mathematical consequences,
we prove the well-posedness of boundary-value problems
subject to dynamic  non-linear and friction-type boundary conditions.
The present work extends the known results
of Laplacian operator to a general elliptic operator in divergence form
with bounded measurable coefficient in the context of diffusion processes.
The motivation comes essentially from the models for the electrical conduction in biological tissues \cite{amar,choilui,cd}.
The construction of generalized solutions is shown via time discretization, following 
the Rothe method \cite{kacur,karel,rek}.

Let $\Omega_1$ and $\Omega_2$ be two disjoint
bounded domains of  $\mathbb R^n(n\geq 2)$ such that
$\bar\Omega=\bar\Omega_1\cup \bar\Omega_2$ is connected with Lipschitz boundary.
Let $\Gamma=\partial\Omega_1\cap\Omega \subset\partial
\Omega_2$ denote a (n-1)-dimensional interface that can include the following
descriptions.
\begin{enumerate}
\item If $\partial\Omega_1 \subset\Omega$ then $\Gamma$ is a closed curve (n=2) or surface
$(n\geq 3)$.
Currently, $\Omega_1$    and $\Omega_2$ are called the inner and the outer domains of  $\Omega$, respectively.
\item If  $\Gamma_1:=\partial\Omega_1\setminus\bar\Gamma=\mbox{int}
(\partial\Omega_1 \cap\partial\Omega)\not=\emptyset$ then
\begin{itemize}
\item if $n=2$, $\Gamma$ is relatively open
(see Fig. \ref{fig1} (a)).

\item
If $n=3,$ $\Omega_1$ stands for a cylindrical-type domain such  that
$\Gamma_1$ represents its top and/or bottom
(see Fig. \ref{fig1} (b)).
\end{itemize}
\item 
The case of $\partial\Omega_1 \cap\partial\Omega\not=\emptyset$ with 
meas$(\partial\Omega_1 \cap\partial\Omega)=0$ 
can be clearly included whenever  $\partial\Omega_2$ is  Lipschitz continuous
(see Fig. \ref{fig1} (c)).
\end{enumerate}
In conclusion, we assume that $\partial\Omega_i$ $(i=1,2)$ are Lipschitz continuous.
The domains  have neither cuts (cracks) nor cusps, and situations as in Fig. \ref{fig1} (d)
are excluded.
\begin{figure}[h]
\includegraphics[width=\linewidth]{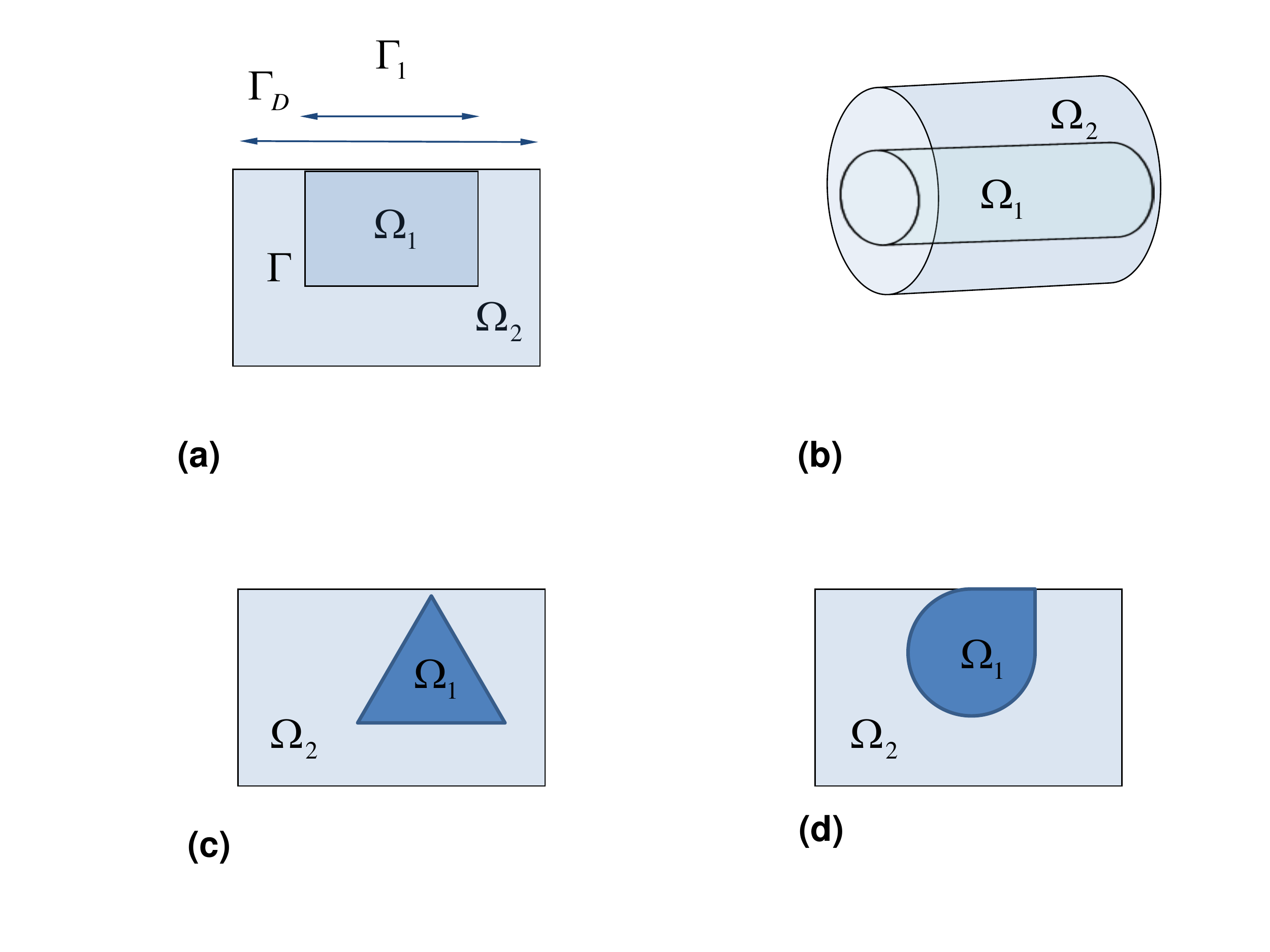}
\caption{The geometry and interface conditions:  2D (a) and 3D (b) models 
 when $\Gamma_1\not=0$; (c) other possible situation;
(d) 2D counterexample.}
\label{fig1}
\end{figure}
Define a relatively open (n-1)-dimensional set $\Gamma_2\subset\partial\Omega_2\setminus \Gamma$, 
 with meas$(\Gamma_2)>0$, 
 and $\Gamma_D=\Gamma_1\cup \Gamma_2$ where we will impose Dirichlet boundary conditions.

Let us introduce the problems under study. For $T>0$,
find $u_i:\Omega_i\times ]0,T[\rightarrow \mathbb R$ satisfying
\begin{equation}\label{delta}
-\nabla\cdot(\sigma_i\nabla u_i)=f_i\mbox{ in }\Omega_i\ (i=1,2).
\end{equation}
The first mathematical interest of this problem is due to the discontinuous coefficient
which reflects the spatial dependence of the conductivity on the electrical conduction in different materials.

On the exterior boundary
$\partial\Omega=(\partial \Omega_2\setminus\Gamma)\cup \Gamma_1, $
 we have homogeneous mixed boundary condition
\begin{equation}\label{mbc}
\nabla u_2\cdot{\bf  n}=0\mbox{ on }\partial\Omega\setminus
\Gamma_D\quad\mbox{ and $u_i=0$ on $\Gamma_D.$}
\end{equation}
On the interface $\Gamma$, we study two different types of dynamic bilateral conditions.
\begin{description}
 
\item[Wentzell-type transmission]

The generalized Wentzell transmission boundary condition is given by
\begin{eqnarray}\label{equal}
u_1=u_2 \quad \mbox{ and }\\
\ [ \sigma\nabla u\cdot{\bf n}]+\beta\Delta u_1
-\alpha\partial_tu_1\in\partial
j(u_1)\mbox{ on }\Sigma:=\Gamma\times ]0,T[,\label{wbc}
\end{eqnarray}
under the initial condition
\begin{equation}
u_1(\cdot,0)=S\mbox{ on }\Gamma\label{ci}
\end{equation}
 where $\alpha$ and $S$  are known functions and $\beta$ is a non-negative constant.
 If $\beta=0$, 
the transmission boundary condition  (\ref{equal})-(\ref{ci}) looks for the
 transmission in a thin (or lower dimensional) porous layer.
Here ${\bf n}$ is the normal unit vector to $\Gamma$ pointing into $\Omega_2$, 
$\partial$ is the subdifferential with  respect to the argument  of the function $j$,
and $[\cdot]$ denotes the jump of a quantity across the interface in direction of $\bf n$,
e.g. $[\sigma\nabla u\cdot{\bf n}]:=\sigma_2\nabla u_2\cdot{\bf
  n}-\sigma_1\nabla u_1\cdot{\bf n}$.
 
\item[Signorini-type transmission]
The transmission that characterizes the boundary thin obstacle problems such as the
semi-permeable membrane  is constituted by
the jump condition
\begin{equation}\label{bcsp}
[\sigma\nabla u\cdot{\bf n}]=g\mbox{ on }\Gamma,
  \end{equation}
and 
the Signorini-type boundary condition
\begin{equation}\label{sbc}
\sigma_2\nabla u_2\cdot{\bf n}-\alpha\partial_t[u]\in\partial
j([u])\mbox{ on }\Sigma:=\Gamma\times ]0,T[,\end{equation}
accomplished with the initial condition
\begin{equation}
[u](\cdot,0)=S\mbox{ on }\Gamma\label{cf}
\end{equation}
 where $g$, $\alpha$,  $j$ and $S$  are known functions \cite{amar}.

\end{description}

The most common application is when $\partial j$ represents the indicatrice Heaviside.
These boundary-value problems also model some of the slip
phenomena observed in contact problems \cite{dl72,rp}.
Other related problems are the unilateral problems \cite{ant}.

The paper is organized as follows. Next Section we set the functional space framework,
the assumptions on the data and main results.
 Sections \ref{s2} and \ref{s3} are devoted to the proofs of existence and uniqueness of weak solutions
 of each problem, namely provided by the Wentzell-type and
Signorini-type transmission, respectively.
These two Sections have similar structures based on the time-discretization technique
and are split into several subsections in order to clarify the exposition.
 In Section \ref{ssasymp}, we show how the unique solution to the  boundary value problem
 provided by a thin porous layer
can be obtained as the limit of perturbed problems.
Finally, some additional regularity is shown in  corresponding 
 Sections \ref{reg2} and \ref{reg1}.
 
\section{Functional space framework and main results}

The data are given under the following regularity assumptions.
Here we assume that 
\begin{equation}\label{defs}
\sigma_i\in L^\infty(\Omega_i):\ 
\exists\sigma_\#,\sigma^\#>0,\quad
\sigma_\#\leq\sigma_i(x)\leq\sigma^\#,\quad\mbox{for a.a.} x\in\Omega_i;
\end{equation}
for $i=1,2$,
\begin{equation}\label{defa}
\alpha\in L^\infty(\Gamma):\ 
\exists\alpha_\#,\alpha^\#>0,\quad
\alpha_\#\leq\alpha(x)\leq\alpha^\#,\quad\mbox{for a.a.} x\in\Gamma;
\end{equation}
and $  j:\mathbb R\rightarrow\mathbb R$ is a  convex and lower semicontinuous function such that
\begin{equation}
 j\geq 0\quad\mbox{ and }\quad j(0)= 0.\label{defj}
\end{equation}

Let us define 
\begin{eqnarray*}
H^1_{\Gamma_D}(\Omega)&=&\{v\in H^1(\Omega):\ v|_{\Gamma_D}=0\};\\
H^1_{\Gamma_i}(\Omega_i)&=&\{v\in H^1(\Omega_i):\ v|_{\Gamma_i}=0\},\qquad (i=1,2).
\end{eqnarray*}

For  a Lipschitz domain $\Omega_1$, the trace operator
$H^1_{\Gamma_1}(\Omega_1)\rightarrow H^{1/2}_{00}(\Gamma)$ has bounded linear right inverse,
that is, for every element  $S$ of the trace space 
\[H ^ {1/2}_{00}(\Gamma)=\{v\in L^2(\Gamma):\ \mbox{its zero extension belongs to }
H^{1/2}(\partial \Omega_1)\}\]
there exists $ u^0_1\in H^1_{\Gamma_1}(\Omega_1)$
such that $u^0_1=S$ on $\Gamma$ \cite{grisv}.
However, the trace mapping considered as a mapping from $H^1_{\Gamma_2}(\Omega_2)$ in
$L^2(\partial\Omega_2)$ is  surjective on $H^{1/2}_{00}(\partial\Omega_2\setminus\bar \Gamma_2)$.

Considering that the Poincar\'e inequality occurs when
$\Gamma_D\cap\partial\Omega_i\not=\emptyset$, for $i=1,2$,
then the above Hilbert spaces are  endowed with the norms
$$\|v\|_{H^1_{\Gamma_i}(\Omega_i)}=\|\nabla
  v\|_{2,\Omega_i}.$$

When $\Gamma_1=\emptyset$ and then we endow $H^1_{\Gamma_1}(\Omega_1)$  with any of the 
equivalent norms
\[\|v\|_{2,\Omega_1}+\|\nabla
  v\|_{2,\Omega_1}\sim  \|v\|_{2,\Gamma}+\|\nabla  v\|_{2,\Omega_1}.\]
Indeed, we recognize that $H^1_{\Gamma_1}(\Omega_1)\equiv H^1(\Omega_1)$ and
 $H ^ {1/2}_{00}(\Gamma)\equiv H ^ {1/2}(\partial\Omega_1)$.
 
\subsection{Wentzell-type transmission}

We can interpret the solutions  $u_i:\Omega_i\times ]0,T[\rightarrow \mathbb R$ $ (i=1,2)$
as the uniquely (almost everywhere) determined function $u:\Omega\times ]0,T[\rightarrow \mathbb R$ such that
$u|_{\Omega_1}=u_1$, $u|_{\Omega_2}= u_2$ and
$u_1=u_2$ on $\Gamma$.

Let us define $H_\beta$ as the Hilbert space
\begin{eqnarray*}
\{v\in H^1_{\Gamma_D}(\Omega):\ v_1=v|_{\Omega_1};\ v_2=v|_{\Omega_2};\
v_1=v_2\mbox{ on }\Gamma
\}&\mbox{ if }&\beta=0;\\
\{v\in H^1_{\Gamma_D}(\Omega):\ v_1=v|_{\Omega_1};\ v_2=v|_{\Omega_2};\
v_1=v_2\mbox{ on }\Gamma; \ \nabla v\in L^2(\Gamma)
\}&\mbox{ if }&\beta>0,
\end{eqnarray*}
endowed with the inner product
\[
(u,v)_\beta=\int_\Omega\nabla u\cdot\nabla vdx+\beta\int_\Gamma\nabla u\cdot\nabla v ds.
\]

\begin{definition} \label{dwt}
We say that a function ${u}\in L^2(0,T;H_\beta)$ is a weak solution to the problem
 (\ref{delta})-(\ref{ci}) 
if  $\partial_t u\in L^2(\Sigma)$ and it satisfies (\ref{ci}) and the variational formulation 
\begin{eqnarray}
\int^T_0\int_\Omega\sigma\nabla{
    u}\cdot\nabla(v-u)dxdt
+\beta\int^T_0\int_\Gamma\nabla u \cdot\nabla(v-u)dsdt+\nonumber\\
+\int^T_0\int_\Gamma\alpha\partial_tu(v-u)dsdt
+\int_0^T\int_\Gamma\{j(v)-j(u)\}dsdt\geq  \nonumber\\
\geq\int^T_0\langle f,v-u\rangle_\Omega dt,\quad\forall v\in
L^2(0,T;H_\beta),\label{pbuw}
\end{eqnarray}
with
$$\sigma=\sigma_1\chi_{\Omega_1}+\sigma_2\chi_{\Omega_{2}}\qquad\mbox{and}\qquad
f=f_1\chi_{\Omega_1}+f_2\chi_{\Omega_{2}}.$$
\end{definition}
 The symbol $\langle\cdot,\cdot\rangle_\Omega$  denotes
the duality pairing $\langle\cdot,\cdot\rangle_{(H_\beta)'\times H_\beta}$.

For $u:\Omega\times ]0,T[\rightarrow \mathbb R$ such that the homogeneous Neumann boundary
condition in (\ref{mbc}) is satisfied,
 the Green formula yields
\[
-\langle\nabla\cdot(\sigma\nabla u), v\rangle_\Omega=\int_\Omega \sigma\nabla u\cdot\nabla
vdx+\langle[\sigma\nabla u\cdot{\bf n}],v\rangle_\Gamma,\qquad \forall v\in H_\beta.\]
Thus, using (\ref{delta}) and (\ref{wbc}) it follows (\ref{pbuw}).

\begin{theorem}
\label{teo2}
Under the assumptions
(\ref{defs})-(\ref{defj}), 
\begin{eqnarray}\label{duow}
\exists { u}^0\in {H_\beta}:\quad
u^0=S\mbox{ on }\Gamma;\\
 \int_\Gamma j(S)ds\leq C(\|S\|_{2,\Gamma}^2+1),\label{js2}
\end{eqnarray}
where $C$ stands for a positive constant,
and $f\in C^{0,1}(0,T;(H_\beta)')$ with 
the Lipschitz constant $d$, that is,
\begin{equation}\label{deffw}
\|f(\tau)-f(t)\|_{(H_\beta)'}\leq d |\tau-t|,\qquad\forall \tau,t\in ]0,T[ ,
\end{equation}
there exists $ u\in L^\infty(0,T;H_\beta)$
 a unique weak solution in accordance to Definition \ref{dwt}.
\end{theorem} 

\begin{remark}
The assumption (\ref{js2}) yields if for instance $j$ verifies 
$j(d)\leq C(d^2+1)$ for all $d\in\mathbb R$.
Notice that (\ref{duow}) guarantees that $S\in L^2(\Gamma)$ is
such that   $\beta\nabla S\in L^2(\Gamma)$. 
\end{remark}

\begin{theorem}\label{prop2}
Let the assumptions of Theorem \ref{teo2} be fulfilled.
 Moreover, if the compatibility condition 
\begin{eqnarray}
\int_\Omega\sigma\nabla u^0\cdot\nabla (v- u^0) dx+\beta\int_\Gamma \nabla u^0\cdot\nabla
(v- u^0) ds+\int_\Gamma \{j(v)-j(S)\}ds\geq\nonumber\\
\geq\langle f(0),v -u^0\rangle_\Omega\qquad\label{wuoi}
\end{eqnarray}
holds for all $ v\in H_\beta,$
 then $\partial_t u\in L^2(0,T;H_\beta)\cap  L^\infty(0,T;L^2(\Gamma))$. In particular, 
$ u\in C([0,T];H_\beta)$.
\end{theorem}

The transmission problem in a thin porous layer, (\ref{delta})-(\ref{ci}) with $\beta=0$,  can be
obtained as the asymptotic limit, when a small parameter $\varepsilon $ goes
to zero, of the following perturbed problem, whenever the interface $\Gamma=\partial\Omega_1
\subset \Omega$, $\Gamma_1=\emptyset$ and  $\Gamma_D=\Gamma_2$,

\noindent $({\bf P}_\varepsilon)$ Find $u_\varepsilon:\Omega=\Omega_1
\cup\overline {S_\varepsilon }\cup\Omega_{2,\varepsilon }\rightarrow\mathbb R $ satisfying
\begin{eqnarray}
-\sigma_1\Delta u_\varepsilon =f_1&\mbox{ in }&\Omega_1;\nonumber\\
-\sigma_2\Delta u_\varepsilon =f_2&\mbox{ in }&\Omega_{2,\varepsilon };\nonumber\\
\varepsilon \gamma\Delta u_\varepsilon 
-{\alpha}\partial_t u_\varepsilon \in\partial
j(u_\varepsilon )&\mbox{ in }&S_\varepsilon \times ]0,T[;\nonumber\\
u_\varepsilon (\cdot,0)=u^0&\mbox{ in }&S_\varepsilon;\label{cis}\\
{[}u_\varepsilon ]=[\sigma\nabla u_\varepsilon \cdot{\bf n} ]
=0&\mbox{ on  }&\Gamma;\nonumber\\
{[}u_\varepsilon ]=[\sigma\nabla u_\varepsilon \cdot{\bf n}]
=0&\mbox{ on  }&\Gamma_\varepsilon :=\partial S_\varepsilon \setminus\Gamma;\nonumber\\
\nabla u_2\cdot{\bf  n}=0&\mbox{ on }&\partial\Omega\setminus\Gamma_2;\nonumber\\
u_2=0&\mbox{ on }&\Gamma_2,\nonumber
\end{eqnarray}
with $S_\varepsilon =\{\xi+\tau{\bf
  n}(\xi):\ \xi\in\Gamma,\ 0<\tau<\varepsilon \gamma(\xi)\}$ where $\gamma\in C^{0,1}(\Gamma)$ such that $
0<\gamma_\#\leq\gamma(\xi)\leq\gamma^\#$ for all $\xi\in\Gamma,$  and $\varepsilon >0$
such that $\overline{S_\varepsilon}\subset\Omega$.

Let us define the Hilbert space
\begin{eqnarray*}
X_\varepsilon =\{v\in H^1_{\Gamma_2}(\Omega_\varepsilon ):\ v_1=v|_{\Omega_1},\ 
v_{S_\varepsilon }=v|_{S_\varepsilon }, \ v_{2,\varepsilon }=v|_{\Omega_{2,\varepsilon }};\\
v_1=v_{S_\varepsilon }\mbox{ on }\Gamma,\ v_{S_\varepsilon }=v_{2,\varepsilon }\mbox{ on }
\Gamma_\varepsilon 
\},
\end{eqnarray*}
where $\Omega_\varepsilon =\Omega_1
\cup S_\varepsilon \cup\Omega_{2,\varepsilon }$.

\begin{prop}
\label{teo3}
Let the assumptions (\ref{defs})-(\ref{defj}), (\ref{duow}),
(\ref{deffw}) and $\beta=0$ be fulfilled, and (\ref{js2}) be replaced by
 $j(d)\leq C(d^2+1)$ for all $d\in\mathbb R$. Then
the  unique solution $ u$ of the problem
(\ref{delta})-(\ref{ci}) in accordance to Theorem \ref{teo2},
under the admissible test function space
${\mathcal X}:=L^2(0,T;H)\cap H^1(0,T;H^{1}_{}(\Omega\setminus \overline{\Omega_1}))$,
is the limit of the sequence of the  unique solutions $ u_\varepsilon $
  to the variational formulation of the perturbed problem $({\bf P}_\varepsilon)$
\begin{eqnarray}
\int^T_0\int_{\Omega_\varepsilon }\sigma_\varepsilon \nabla
u_\varepsilon \cdot\nabla(v-u_\varepsilon )dxdt+\int^T_0\int_{
S_\varepsilon }{\alpha\over\varepsilon   \gamma}\partial_t
u_\varepsilon (v-u_\varepsilon )dxdt+\nonumber\\+\int^T_0\int_{S_\varepsilon }{1\over
\varepsilon \gamma}\{j(v)-j(u_\varepsilon )\}dxdt
\geq\int^T_0\langle f_\varepsilon,v-u_\varepsilon\rangle_{\Omega_\varepsilon }dt,\ \forall
v\in L^2(0,T;X_\varepsilon ),\label{ueps}
\end{eqnarray}
 with (\ref{cis}),
$\sigma_\varepsilon =\sigma_1\chi_{\Omega_1}+\chi_{S
_\varepsilon }+\sigma_2\chi_{\Omega_{2,\varepsilon }}$ and
$f_\varepsilon =f_1\chi_{\Omega_1}+f_2\chi_{\Omega_{2,\varepsilon }}$.
\end{prop}

\subsection{Signorini-type transmission}

Here, we keep the notation of jump $[v]=v_2-v_1$ for any vector ${\bf v}=(v_1,v_2)$.
However, 
in order to differentiate this case from the above, let us set every vector by boldface.
In general if $v_1\not =v_2$ on $\Gamma$, their weak derivatives do not exist.
Let us define the Hilbert space 
\[
{\bf V}=\{{\bf v}=(v_1,v_2):\ v_1\in
H^1_{\Gamma_1}(\Omega_1);\ v_2\in H^1_{\Gamma_2}(\Omega_2)\}
\hookrightarrow L^2(\Omega_1)\times L^2(\Omega_2)\]
endowed with the norm (cf. Lemma \ref{lpoin})
\[
\|{\bf v}\|_{\bf V}=\|\nabla v_1\|_{2,\Omega_1}+\|\nabla v_2\|_{2,\Omega_2}+
\| [v]\|_{2,\Gamma}.
\]
For ${\bf v}\in {\bf V}$, 
${\bf v}|_\Gamma\in H^{1/2}_{00}( \Gamma)\times
H^{1/2}_{00}(\partial\Omega_2\setminus\bar \Gamma_2)$.

\begin{definition} \label{dsp}
We say that a function ${\bf u}=(u_1,u_2)
\in L^2(0,T;{\bf V})$ is a weak solution to the problem (\ref{delta})-(\ref{mbc}) 
with (\ref{bcsp})-(\ref{cf}) 
if $\partial_t [u]\in
L^2(\Sigma)$ and it satisfies (\ref{cf}) and the variational formulation
\begin{eqnarray}
\int^T_0\int_\Omega\sigma\nabla{\bf
    u}\cdot\nabla({\bf v}-{\bf u})dxdt+\int^T_0\langle g,v_1-u_1\rangle_\Gamma dt+\nonumber\\
+\int^T_0\int_\Gamma\alpha\partial_t[u]([v]-[u])dsdt
+\int_0^T\int_\Gamma \{j([v])-j([u])\}dsdt\geq \nonumber\\
\geq\int^T_0\langle {\bf f},{\bf v}-{\bf u}\rangle_\Omega dt ,\quad\forall {\bf v}\in
L^2(0,T;{\bf V}),\label{pbu}\end{eqnarray}
with
\[
\sigma=\sigma_1\chi_{\Omega_1}+\sigma_2\chi_{\Omega_{2}}\qquad\mbox{and}\qquad
{\bf f}=(f_1,f_2).\]
\end{definition}
 Here, we use the same notation
$\langle\cdot,\cdot\rangle_\Omega$ to denote
the duality pairing $\langle\cdot,\cdot\rangle_{{\bf V}'\times {\bf V}}$.
The symbol $\langle\cdot,\cdot\rangle_\Gamma$ stands for
the duality pairing $\langle\cdot,\cdot\rangle_{Y'\times Y}$, using
 the notation $Y=H^{1/2}_{00}(\Gamma)$.

For ${\bf u}=(u_1,u_2)$ such that the homogeneous Neumann boundary
condition in (\ref{mbc}) is satisfied,
 the Green formula yields
\[
-\langle\nabla\cdot(\sigma\nabla {\bf u}), {\bf v}\rangle_\Omega=
\int_\Omega \sigma\nabla {\bf u}\cdot\nabla
{\bf v}dx+\langle[\sigma\nabla {u}\cdot{\bf n}],v_1\rangle_\Gamma+\langle\sigma_2\nabla
u_2\cdot{\bf n},[v]\rangle_\Gamma,\]
for all ${\bf v}\in {\bf V}$.
Thus, using (\ref{delta}) and (\ref{bcsp})-(\ref{sbc})
it follows (\ref{pbu}).
 
\begin{theorem}\label{teo1}
Assuming (\ref{defs})-(\ref{defj}), (\ref{js2}),
${\bf f}$ and $g$ are Lipschitz functions in the following sense:
 there exist two positive constants $d_1$ and $d_2$ such that
\begin{eqnarray}
\|{\bf f}(\tau)-{\bf f}(t)\|_{{\bf V}'}\leq d_1 |\tau-t|&&\label{deff}\\\label{defg}
\|g(\tau)-g(t)\|_{Y'}\leq d_2|\tau-t|,&&\qquad\forall \tau,t\in ]0,T[ .
\end{eqnarray}
and
\begin{equation}\label{duo}
\exists {\bf u}^0\in {\bf V}:\quad
[u^0]=S\mbox{ on }\Gamma,
\end{equation}
there exists $ {\bf u}\in L^\infty(0,T;{\bf V})$
 a unique weak solution in accordance to Definition \ref{dsp}.
\end{theorem}
\begin{remark}
The  assumption (\ref{duo}) implies that
\[\|S\|_{2,\Gamma}\leq\|[u^0]\|_{2,\Gamma}\leq
\|{\bf u}^0\|_{\bf V}.\]
\end{remark}
\begin{theorem}\label{prop1}
Let the assumptions of Theorem \ref{teo1} be fulfilled.
 Moreover, if the compatibility condition 
\begin{equation}\label{uoi}
\int_\Omega\sigma\nabla {\bf u}^0\cdot\nabla ({\bf v}- {\bf u}^0) dx+\langle g(0)
,v_1- u^0_1\rangle_\Gamma
+\int_\Gamma \{j([v])-j(S)\}ds\geq\langle {\bf f}(0),{\bf v} -{\bf u}^0\rangle
_\Omega,
\end{equation}
holds for all ${\bf v}\in {\bf V}$,
then $\partial_t {\bf u}\in L^2(0,T;{\bf V})\cap  L^\infty(0,T;L^2(\Gamma))$.  In particular, 
${\bf u}\in C([0,T];{\bf V})$.
\end{theorem}

\section{Proof of Theorem \ref{teo2}}
\label{s2}

\subsection{Discretization in time}
\label{time}

In the following we use similar arguments from the methods
 described in \cite{karel}.
We decompose the time interval $I=[0,T]$ into $m$
subintervals $I_{i,m}=[t_{i,m},t_{i+1,m}]$ of size
$h=T/m,$ $i\in\{0,1,\cdot\cdot\cdot,m-1\},$ $
m\in\mathbb N.$ We define, for all
$i\in\{0,1,\cdot\cdot\cdot,m-1\},$ $u^{i+1}=u(t_{i+1,m
})$ as solutions
given at the following Proposition.

\begin{prop}\label{propuiw}
 Let  $i\in \{0,1,\cdot\cdot\cdot,m-1\}$ be fixed,
 $u^i\in L^2(\Gamma)$, and
\[f^{i+1}=
f(t_{i+1,m})\in (H_\beta)'.\]
Then there exists $u^{i+1}\in H_\beta$ a solution to the problem
\begin{eqnarray}
\int_\Omega\sigma\nabla u^{i+1}\cdot\nabla(v-u^{i+1})dx
+\beta\int_\Gamma \nabla u^{i+1}\cdot\nabla(v-u^{i+1})ds+\nonumber\\
+\int_\Gamma{\alpha\over  h}u^{i+1}\left(v-u^{i+1}\right)ds
+\int_\Gamma\{j(v)-j(u^{i+1})\}ds\geq\nonumber\\
\geq\langle
f^{i+1},v-u^{i+1}\rangle_\Omega+\int_\Gamma{\alpha\over  h}u^i\left(
v-u^{i+1}\right)ds,\quad \forall v\in H_\beta.\label{wui}\end{eqnarray}
\end{prop}
\begin{proof}
The existence of a solution to (\ref{wui}) is deduced from
the general theory on maximal monotone mappings applied to
elliptic variational inequalities \cite[pp. 874-875, 892-893]{z}.
Indeed, the mapping $A:H_\beta\rightarrow (H_\beta)'$ defined by
$$\langle Au, v\rangle=\int_\Omega \sigma\nabla u\cdot\nabla vdx
+\beta\int_\Gamma\nabla u\cdot\nabla v ds+\int_\Gamma {\alpha\over h}uvds$$
is single-valued, linear and hemicontinuous;
the mapping $\varphi:H_\beta\rightarrow [0,+\infty]$
defined by
$$\varphi(v)=\left\{\begin{array}{ll}
\int_\Gamma j(v)ds,&\mbox{ if }
  j(v)\in L^1(\Gamma)\\
+\infty,&\mbox{ otherwise}
\end{array}\right.$$
is convex, lower semicontinuous and $\varphi\not\equiv +\infty$;
and the coercivity condition is valid
$$\langle Au,u\rangle+\varphi(u)
=\int_\Omega \sigma|\nabla u|^2dx+\beta\int_\Gamma |\nabla u|^2ds
\geq  \min\{\sigma_\#,1\}\|u\|_{H_\beta}^2,$$
under the assumptions (\ref{defs})-(\ref{defj}).
Then, for $b\in ({H_\beta})'$ such that
$$\langle b, v\rangle=-\langle f^{i+1}, v\rangle_\Omega -\int_\Gamma {\alpha\over h}u^ivds,$$
the variational inequality (\ref{wui}) has a unique weak solution
$u=u^{i+1}\in H_\beta.$
\end{proof}
\begin{remark}
Since $u^0=S$ on $\Gamma$ means that $u^0\in L^2(\Gamma)$, then Proposition \ref{propuiw}
guarantees the existence of $u^1\in V$ and consequently $u^1\in L^2(\Gamma)$.
Therefore,  Proposition \ref{propuiw} successively
guarantees the existence of $u^{i+1}\in V$ for every $i=1,\cdots,m-1$.
\end{remark}

\subsection{Existence of a limit $u$}

\begin{prop}\label{tildew}
For all  $i\in \{0,1,\cdots,m-1\}$, the estimate holds
 \begin{equation}\label{dui1w}
\alpha_\#\|u^{i+1}\|^2_{2,\Gamma}\leq\max \{{1\over\sigma_\#},1\}
\|f\|^2_{L^2(0,T;(H_\beta)')}+ \alpha^\#
\|S\|_{2,\Gamma}^2.\end{equation}
Moreover, if
$\{\tilde u_m\}_{m\in\mathbb N}$ is the sequence defined by the step functions
$\tilde u_m:I\rightarrow H_\beta$
$$\tilde u_m (t)=\left\{\begin{array}{ll}
u^{1}&\mbox{ for }t=0\\
 u^{i+1}&\mbox{ in }]t_{i,m},t_{i+1,m}]
\end{array}\right.$$
then there exists $u$ such that
$$\tilde u_m\rightharpoonup u\mbox{ in }L^2(0,T;H_\beta).$$
\end{prop}
\begin{proof}
Choosing $v=0$ as a test function in (\ref{wui}), we get
$$\int_\Omega\sigma|\nabla u^{i+1}|^2dx+\beta\int_\Gamma|\nabla u^{i+1}|^2ds+\int_\Gamma{\alpha\over  h}(u^{i+1})^2ds\leq\langle f^{i+1},
u^{i+1}\rangle_\Omega+\int_\Gamma {\alpha\over h}u^iu^{i+1}ds,$$
 for all  $i\in \{0,1,\cdot\cdot\cdot,m-1\}$.
Then it follows
$$\min\{\sigma_\#,1\} \| u^{i+1}\|^2_{H_\beta}+\int_\Gamma{\alpha\over  h}(u^{i+1})^2ds\leq
\max \{{1\over\sigma_\#},1\}\|f^{i+1}\|^2_{({H_\beta})'}+\int_\Gamma{\alpha\over    h}(u^i)^2ds.$$
Summing on $k=0,...,i$, it follows 
$$
\min\{\sigma_\#,1\} h\sum_{k=0}^i  \| u^{k+1}\|^2_{H_\beta}+\alpha_\#\|u^{i+1}\|^2_{2,\Gamma}
\leq \max \{{1\over\sigma_\#},1\}h\sum_{k=1}^{i+1}\|f^k
\|^2_{({H_\beta})'}+{\alpha^\#} \|S\|_{2,\Gamma}^2.
$$

Consequently, we get (\ref{dui1w}) and, for
$i=m-1$,
 \begin{equation}\label{ui1w}
\min \{\sigma_\#,1\}\|\tilde u_m\|^2_{L^2(0,T;{H_\beta})}\leq\max\{{1\over\sigma_\#},1\}
\|f\|^2_{L^2(0,T;({H_\beta})')}+ \alpha^\#
\|S\|_{2,\Gamma}^2.\end{equation}
Thus we can extract a subsequence, still denoted by $\tilde
u_m,$ weakly convergent to $u\in L^2(0,T;H_\beta).$
\end{proof}
Next, let us study the discrete derivative with respect to $t$ at the time $t=t_{i+1}$:
\[Z^{i+1}:={u^{i+1}-u^{i}\over h}.
\]
\begin{prop}\label{zz}
 Let
 $Z_m:[0,T[\rightarrow L^2(\Omega)$   be defined by
$$Z_m (t)=\left\{\begin{array}{ll}
Z^{1}&\mbox{ for }t=0\\&\\
Z^{i+1}&\mbox{ in }]t_{i,m},t_{i+1,m}]
\end{array}\right.\mbox{ in }\Omega.$$
If the assumptions (\ref{defs})-(\ref{defj}) and (\ref{duow})-(\ref{deffw}) are fulfilled,
then the estimate holds
\begin{equation}\label{ui2w}
\|\tilde u_m\|^2_{L^\infty(0,T;{H_\beta})}+\|Z_m\|^2_{2,\Sigma}\leq
C(\|f\|^2_{L^2(0,T;(H_\beta)')}+\|u^0\|_{H_\beta}^2).
\end{equation}
Hence, we can extract a subsequence, still denoted by $Z_m,$ weakly
convergent to $Z\in L^2(\Sigma).$
\end{prop}
\begin{proof}
 For a fixed $t$, there exists
$i\in\{0,\cdot\cdot\cdot,m-1\}$ such that  $t\in]t_{i,m};t_{i+1,m}]$.
Choosing $v=u^i$ as a test function in (\ref{wui}), we have
\begin{eqnarray*}
\int_\Omega\sigma\nabla u^{i+1}\cdot\nabla(u^{i+1}-u^i)dx
+\beta\int_\Gamma \nabla u^{i+1}\cdot\nabla(u^{i+1}-u^i)ds+\\
+\int_\Gamma{\alpha\over  h}(u^{i+1}-u^i)^2ds 
+\int_\Gamma j(u^{i+1})ds\leq
 \int_\Gamma j(u^i)ds+\langle f^{i+1},u^{i+1}-u^i\rangle_\Omega.
\end{eqnarray*}
In order to sum the above expression  on $k=0,...,i$,
 consider the relation $2(a-b)a=a^2+(a-b)^2-b^2$ to obtain
\begin{eqnarray*}
\sum_{k=0}^i\int_\Omega\sigma\nabla u^{k+1}\cdot\nabla(u^{k+1}-u^k)dx&=&
{1\over 2}\int_\Omega\sigma|\nabla u^{i+1}|^2dx-{1\over 2}\int_\Omega\sigma
|\nabla u^0|^2dx+\\&+&{1\over 2}\sum_{k=0}^i
\int_\Omega\sigma|\nabla(u^{k+1}-u^k)|^2dx;\\
\sum_{k=0}^i\int_\Gamma\nabla u^{k+1}\cdot\nabla(u^{k+1}-u^k)ds&=&
{1\over 2}\int_\Gamma|\nabla u^{i+1}|^2ds-{1\over 2}\int_\Gamma
|\nabla u^0|^2ds+\\&+&{1\over 2}\sum_{k=0}^i
\int_\Gamma|\nabla(u^{k+1}-u^k)|^2ds.
\end{eqnarray*}

Now, using the assumptions (\ref{defs})-(\ref{defj})  we find
\begin{eqnarray}
{\min \{\sigma_\#,1\}\over 2}\|u^{i+1}\|^2_{H_\beta}+
\alpha_\#\sum_{k=0}^i h\int_\Gamma\left({u^{k+1}-u^k\over
h}\right)^2ds\leq {\sigma^\#
\over 2}\|\nabla u^0\|_{2,\Omega}^2+\nonumber\\+{\beta
\over 2}\|\nabla u^0\|_{2,\Gamma}^2+
\int_\Gamma j(S)ds
-\langle f^1,u^0\rangle_\Omega-\sum_{k=1}^i\langle f^{k+1}-f^k,u^k\rangle_\Omega+\langle
f^{i+1},u^{i+1}\rangle_\Omega.\label{dtuiw}
\end{eqnarray}

 By (\ref{deffw}) it follows
 \begin{eqnarray*}
 \sum_{k=1}^i\langle f^{k+1}-f^k,u^k\rangle_\Omega\leq d h
 \sum_{k=1}^i\|u^k\|_{H_\beta}.
 \end{eqnarray*}
Therefore, inserting the above inequality in (\ref{dtuiw}) and applying (\ref{ui1w}), it results
(\ref{ui2w}).
\end{proof}

From the Rothe function defined by
$$u_1(x,t)=u^0(x)+t\ {u^{1}(x)-u^0(x)\over h}\mbox{ in } I_{0,1}=I,$$
consider the following definition.
\begin{definition} \label{defro}
We say that  $\{u_m\}_{m\in\mathbb N}$ is the Rothe sequence if
$$u_m(x,t)=u^{i}(x)+(t-t_{i,m}){u^{i+1}(x)-u^{i}(x)\over
  h{}}\mbox{ in } I_{i,m},$$
for all  $i\in \{0,1,\cdots,m-1\}$.
  \end{definition} 
  
\begin{prop}\label{duuw}
If $Z$ satisfies Proposition \ref{zz}, then
$$\partial_tu=Z\mbox{ in }L^2(\Gamma),\mbox{ for almost all }t\in
I.$$
\end{prop}
\begin{proof}  For a fixed $t$, there exists
$i\in\{0,\cdot\cdot\cdot,m-1\}$ such that  $t\in]t_{i,m};t_{i+1,m}]$.
Thus we
obtain
$$\int^t_0 Z_m(\tau)d\tau=\sum_{k=0}^{i-1}\int_{kh}^{(k+1)h}
{u^{k+1}-u^k\over h}d\tau+\int_{ih}^{t}
{u^{i+1}-u^i\over h}d\tau\mbox{ in }\Omega.$$
Because  there exists $w\in C([0,T];{L^2(\Gamma))}$ such that
$$(w(t),v)=\int^t_0(Z(\tau),v)d\tau,\qquad\forall v\in L^2(\Gamma),$$
let us consider Definition \ref{defro} 
on $\Gamma.$ Thus we have
$\int^t_0 Z_m(\tau)d\tau=u_m(t)-S$
and from the Riesz theorem we get
$$(u_m(t)-S
,v)=\int^t_0(Z_m(\tau),v)d\tau,\quad\forall v\in L^2(\Gamma).$$
Indeed, the right hand side of the above equation is
a bounded linear functional in $L^2(\Gamma)$, representable thus
 (uniquely) by the element $u_m(t)-S$ from $L^2(\Gamma).$

Then it follows
\begin{equation}\label{leb}
\lim_{m\rightarrow+\infty}\left(u_m(t)-S-w(t),v\right)=
\lim_{m\rightarrow+\infty}\int^t_0(Z_m(\tau)-Z(\tau),v)d\tau=0.
\end{equation}

Let us prove that
 the norms of the  functions $u_m$ are uniformly bounded with
respect to $t\in I$ and $m$. From  the estimates
(\ref{dui1w}) independent on $i$ and $m$, and
considering 
$$\|u_m(t)\|_{2,\Gamma }=\|u^{i}\left(1+{t-t_{i,m}\over
  h{}}\right)+u^{i+1}{t-t_{i,m}\over h{}}\|_{2,\Gamma}$$
then, we get
$$\|u_m \|^2_{L^\infty(0,T;L^2(\Gamma))}\leq
C(\|f\|^2_{L^2(0,T;(H_\beta)')}+\|S\|^2_{2,\Gamma}).$$

Hence,
 the Lebesgue  Dominated Convergence Theorem can be applied in (\ref{leb}) giving
$$\lim_{m\rightarrow+\infty} \int^T_0(u_m(t)-S-w(t),v)dt=0,\qquad\forall v\in L^2(\Gamma).$$

In the same manner this result can be derived for the case when $v(t)$
is a piecewise constant function of $t \in I.$
Since these functions are dense in $L^2(\Sigma),$ it remains valid for
every function $v\in L^2(\Sigma).$ From the uniqueness of the weak
limit, we conclude
$$u(t)-S=\int^t_0 Z(\tau)d\tau,$$
which corresponds to the claim.
\end{proof}

\subsection{Passage to the limit on $m\rightarrow +\infty$}
\label{passw}

Denoting $f_m(t)=f^{i+1}$  for
$t\in ]t_{i,m},t_{i+1,m}]$ and $i\in\{0,\cdot\cdot\cdot,m-1\}$, we have
\begin{eqnarray*}
\int_Q\sigma\nabla\tilde  u_m\cdot\nabla vdxdt+\beta\int_\Sigma\nabla\tilde  u_m\cdot\nabla vdsdt+
\int_\Sigma\alpha Z_mvdsdt+\\+\int_\Sigma j(v)dsdt
\geq\int_Q
\sigma|\nabla\tilde  u_m|^2dxdt+\beta\int_\Sigma|\nabla\tilde  u_m|^2dsdt+\\
+\int_\Sigma\alpha Z_m \tilde
u_mdsdt+\int_\Sigma j(\tilde  u_m)dsdt+\int_0^T\langle f_m,v-\tilde  u_m\rangle_\Omega dt.
\end{eqnarray*}

From Propositions \ref{tildew} and \ref{zz}
to pass to the limit the above inequality
and recalling the weak lower semicontinuity property for the first and second 
terms on the right hand side of the above inequality,
it remains to prove that
$$\tilde u_m\rightarrow u\mbox{ in }L^2(\Sigma).$$

Taking $\tilde u_m-u=\tilde u_m-u_m+u_m-u$
first let us prove that
$$\tilde u_m-u_m\rightarrow 0\mbox{ in }L^2(\Sigma).$$
Since we have $0<t-t_{i,m}\leq h\mbox{ in }]t_{i,m};t_{i+1,m}]$ we
obtain
$$\|\tilde u_m(t)-u_m(t)\|_{2,\Gamma}=\|Z_m\|_{2,\Gamma}(h-(t-t_{i,m}))<
h\|Z_m\|_{2,\Gamma}$$
and from (\ref{ui2w}) then it follows
$$\|\tilde u_m-u_m\|_{2,\Sigma}\leq{CT\over
  {m}}(\|f\|^2_{L^2(0,T;(H_\beta)')}+\|u^0\|_{H_\beta}^2)^{1/2}\rightarrow
  0.$$

Secondly the Rothe sequence $\{u_m\}$ is bounded in $L^2(0,T;H_\beta)$,
and, from Proposition \ref{duuw},
the functions $\partial_tu_m$ are bounded in
$L^2(\Sigma)$ then, for a
subsequence still denoted by $u_m,$ the strong convergence holds
$$u_m\rightarrow u\mbox{ in } L^2(\Sigma).$$

Then it results
$$\int^T_0\int_\Gamma Z_m\tilde
u_mdsdt\rightarrow\int^T_0\int_\Gamma Uudsdt=
\int^T_0\int_\Gamma\partial_tuudsdt.$$

Therefore we are in the conditions to pass to the limit 
concluding the weak formulation (\ref{pbuw}).

From the standard technique to prove uniqueness of solution, the
solution $u$ to (\ref{pbuw}) with (\ref{cf})
 is unique. Then the whole sequence
$\{\tilde u_m\}$ converges *-weakly to $u\in L^\infty(0,T;H_\beta).$

\section{Regularity in time}
\label{reg2}

{\sc Proof of Theorem \ref{prop2}.}  The proof follows the time discretization argument
as in Theorem  \ref{teo2},
considering the existence of the integral inequality (\ref{wui}).
Choosing
$v=(u^{i+1}+u^i)/2$ as a test function in (\ref{wui}) for the solutions 
$u^{i+1}$ and $u^i$, summing the consecutive integral inequalities,
 and dividing by $h$,  we deduce
\begin{eqnarray*}
\int_\Omega h\sigma|\nabla Z^{i+1}|^2dx+h\beta\int_\Gamma|\nabla Z^{i+1}|^2ds
+\int_\Gamma\alpha(Z^{i+1}-Z^i)Z^{i+1}ds\leq\\ \leq
\langle f^{i+1}-f^i,Z^{i+1}\rangle_\Omega\end{eqnarray*}
taking the convexity of $j$ into account.
Applying the assumptions (\ref{defs}) and (\ref{deffw}),
it results
$$\min\{\sigma_\#,1\} h\| Z^{i+1}\|^2_{H_\beta}+\int_\Gamma\alpha(Z^{i+1}-Z^i)Z^{i+1}ds\leq
d h\|Z^{i+1}\|_{H_\beta}.$$

Considering the relation
$2(a-b)a=a^2+(a-b)^2-b^2$, to $a=Z^{i+1}$ and $b=Z^i,$
and summing on $k=1,\cdots,i$ ($i\in\{1,\cdots,m-1\}$) we obtain
\begin{eqnarray*}
\min\{\sigma_\#,1\}\sum^{i}_{k=1}h\|  Z^{k+1}\|^2_{H_\beta}+\alpha_\#\|  Z^{i+1}\|^2_{2,\Gamma}\leq
2\int_\Gamma\alpha\left({u^1-S\over  h}\right)^2ds+\\+d^2
\max\{{1\over \sigma_\#},1\}\sum^{i}_{k=0}h.
\end{eqnarray*}
Notice that $mh=T$.

Let us determine the estimate for the first term on the right hand
side of the above inequality.
Rewrite the integral inequality (\ref{wui}) for $i=0$ in the form
\begin{eqnarray*}
\int_\Omega\sigma\nabla(u^1-u^0)\cdot\nabla (v-u^1)dx+
\int_\Omega\sigma\nabla u^0\cdot\nabla (v-u^1)dx+\\+
\beta\int_\Gamma \nabla(u^1-u^0)\cdot\nabla (v-u^1)ds+
\beta\int_\Gamma \nabla u^0\cdot\nabla (v-u^1)ds+\\
+\int_\Gamma\alpha{u^1-S\over  h}(v-u^1)ds
  +\int_\Gamma\{j(v)-j(u^1)\}ds
\geq\langle f^1-f(0),  v-u^1\rangle_\Omega+\\+
\langle f(0
) ,v-u^1\rangle_\Omega,\end{eqnarray*}
for all $v\in V$, and in particular $v=u^0$.
Thus, we apply the assumption (\ref{wuoi}) with $v=u^1$ and divide by $h$ we deduce
\begin{eqnarray*}{\sigma_\#\over 2 h}\int_\Omega|\nabla(u^1-u^0)|^2dx+{\beta\over
 2 h}\int_\Gamma|\nabla(u^1-u^0)|^2ds+\int_\Gamma\alpha\left({u^1-S\over h}
\right)^2ds\leq\\ \leq
{C\over 2 h}\|f^1-f(0)\|_{(H_\beta)'}^2.
\end{eqnarray*}
 Then, using (\ref{deffw}), we have
$$\int_\Gamma{\alpha}
\left|{u^1-S\over h}\right|^2ds\leq Chd^2<C.$$

Since the above regularity estimates are independent on $m$
 the proof of the passage to the limit is similar to the one of Section \ref{s2}.
Moreover, the uniqueness of the weak solution implies that
 the weak solution is the strong solution in the sense
$ u\in C([0,T];H_\beta)$ by appealing to the Aubin-Lions Theorem.

\section{Proof of Proposition \ref{teo3}}
 \label{ssasymp}
 
\subsection{Existence of $u_\varepsilon$}
\label{exue}

The time discretization described in Section \ref{time} reads, for the perturbed problem, as
\begin{eqnarray}
\int_{\Omega_\varepsilon }\sigma_\varepsilon \nabla u^{i+1}\cdot\nabla(v-u^{i+1})dx
+\int_{S_\varepsilon }{\alpha\over \varepsilon  h\gamma}(
u^{i+1}-u^i)\left(v-u^{i+1}\right)dx+\nonumber\\
+\int_{S_\varepsilon }{1\over\varepsilon \gamma}\{j(v)-j(u^{i+1})\}dx\geq\langle
f^{i+1},v-u^{i+1}\rangle_{\Omega_\varepsilon },\quad \forall v\in X_\varepsilon .\label{tthin}
\end{eqnarray}

 The existence and uniqueness of 
a solution $u^{i+1}_\varepsilon\equiv u^{i+1}\in X_\varepsilon $
is due to standard results for elliptic variational
inequalities as in the proof of Proposition \ref{propuiw}
(cf. \cite{l}). Indeed, the
bilinear symmetric form
$$a(u,v)=\int_{\Omega_\varepsilon }\sigma_\varepsilon \nabla u\cdot \nabla
vdx+\int_{S_\varepsilon }{\alpha\over \varepsilon  h\gamma}  uvdx$$
is coercive in the following sense
$$a(u,u)\geq \min\{1,\sigma_\#\}\|\nabla u\|^2_{2,\Omega_\varepsilon }
+{\alpha_\#\over  \varepsilon  h \gamma^\#}\|u\|^2_{2,S_\varepsilon }.$$

Now taking first $v=0$ in (\ref{tthin}),
analogously to the proof of Proposition \ref{tildew},
 we get the estimates
\begin{eqnarray}
{\alpha_\#\over \varepsilon\gamma^\#}\|  u^{i+1}\|^2_{2,S_\varepsilon }\leq 
{\alpha^\#\over\varepsilon \gamma_\#}
\|u^0\|^2_{2,S_\varepsilon }+\max \{{1\over\sigma_\#},1\}
\|f_\varepsilon\|^2_{L^2(0,T;(X_\varepsilon)')};\nonumber\\\label{cotat1}
\min\{\sigma_\#,1\}\int^T_0\|\tilde u_{m}\|^2_{X_\varepsilon }dt\leq 
{\alpha^\#\over \varepsilon  \gamma_\#}
\|u^0\|^2_{2,S_\varepsilon }+\max \{{1\over\sigma_\#},1\}
\|f_\varepsilon\|^2_{L^2(0,T;(X_\varepsilon)')}.
\end{eqnarray}

Next taking $v=u^i$ in (\ref{tthin}) and arguing as  the proof of Proposition \ref{zz},
we obtain
\begin{eqnarray*}\min\{1,\sigma_\#\}
\|\nabla u^{i+1}\|^2_{2,\Omega_\varepsilon }+{\alpha_\#h\over
  \varepsilon  \gamma^\#}\sum_{k=0}^i\|Z^{k+1}\|^2_{2,S_\varepsilon }\leq\\
\leq\int_{S_\varepsilon }{1\over\varepsilon \gamma_\#}
j(u^0)dx+C(\|\nabla u^0\|^2_{2,\Omega_\varepsilon }
+
\|f_\varepsilon\|^2_{L^2(0,T;(X_\varepsilon)')}+{1\over \varepsilon  }
\|u^0\|^2_{2,S_\varepsilon }).
\end{eqnarray*}
Thus applying (\ref{js2}) it results that $\tilde u_m$ and $Z_m$ are uniformly
bounded in  $ L^\infty(0,T; X_\varepsilon)$ and  $ L^2(S_\varepsilon\times ]0,T[)$, respectively.
Therefore the existence of a solution $u\in L^2(0,T; X_\varepsilon)$ to
(\ref{ueps}) can be done by similar arguments of passage to the limit as in the
 proof of Theorem \ref{teo2} (cf. Section \ref{passw}).

\subsection{Passage to the limit on $\varepsilon$}

In order to let $\varepsilon \rightarrow 0$, 
 we utilize the following equivalent
variational inequalities to (\ref{ueps})
 and (\ref{pbuw}) with $\beta=0$, respectively, 
\begin{eqnarray}
\int^T_0\int_{\Omega_\varepsilon }\sigma_\varepsilon \nabla
u_\varepsilon \cdot\nabla(v-u_\varepsilon )dxdt+\int^T_0\int_{S_\varepsilon }{\alpha\over
\varepsilon \gamma}{\partial_t}v(v-u_\varepsilon )dxdt+\nonumber\\
+\int_{S_\varepsilon }{\alpha\over 2\varepsilon   \gamma}|v(0)-u^0|^2dx
+\int^T_0\int_{S_\varepsilon }{1\over\varepsilon \gamma}
\{j(v)-j(u_\varepsilon )\}dxdt\geq\nonumber\\
\geq\int^T_0\langle
f_\varepsilon,v-u_\varepsilon\rangle_{\Omega_\varepsilon }dt,
\quad\forall v\in{\mathcal
X}_\varepsilon :=L^2(0,T;X_\varepsilon )\cap H^1(0,T;H^1(S_\varepsilon ));\label{wteo4}
\end{eqnarray}
and
\begin{eqnarray*}
\int^T_0\int_\Omega \sigma\nabla
u\cdot\nabla(v-u)dxdt+\int^T_0\int_\Gamma\alpha\partial_t
v(v-u)dsdt+\int_\Gamma{\alpha\over 2} |v(0)-u^0|^2ds+\nonumber\\
+\int^T_0\int_\Gamma\{j(v)-j(u)\}dsdt\geq\int^T_0\langle
f,v-u\rangle_\Omega dt,\quad\forall v\in{\mathcal X}.\qquad
\end{eqnarray*}

Let $u_\varepsilon $ be the solution of
(\ref{ueps}), or equivalently (\ref{wteo4}), satisfying  (\ref{cis}).
By appealing to Section \ref{exue} we have
$$\|u_\varepsilon \|_{L^\infty(0,T;L^2(S_\varepsilon ))}\leq C(\|
u^0\|_{2,\Omega}+ \|f\|_{L^2(0,T;H')}).$$

Using the result (cf. \cite{cr89})
$${1\over\varepsilon }\|u^0\|^2_{2,S_\varepsilon }\leq C( 
\|u^0\|^2_{2,\Gamma}+\varepsilon \|\nabla u^0\|^2_{2,S_\varepsilon })$$
in the estimate (\ref{cotat1}) it follows
$$\|u_\varepsilon \|_{L^2(0,T;H^1(\Omega_\varepsilon ))}\leq C(\|
u^0\|_{H}+ \|f\|_{L^2(0,T;H')}).$$

Thus there exists a subsequence $\varepsilon \rightarrow 0$ and a function
$u\in L^\infty(0,T;L^2(S_\varepsilon ))\cap  L^2(0,T;H^1(\Omega_\varepsilon ))$ such that
\begin{eqnarray}\label{linf}
u_\varepsilon \rightharpoonup u\quad\mbox{ *-weakly in }
L^\infty(0,T;L^2(S_\varepsilon ));\\
u_\varepsilon \rightharpoonup u\quad\mbox{ weakly in }
L^2(0,T;H^1(\Omega_\varepsilon )).
\end{eqnarray}

Next  we recall the following lemma
which is an extension the one proved in \cite{cr89,cr90}.
\begin{lemma}\label{ljf}
${\bf a)}$ For any function $w\in W^{1,1}(\Omega\setminus \overline{\Omega_1})$
we have 
$$\int_{S_\varepsilon }{w\over \varepsilon \gamma}dx\rightarrow\int_\Gamma wds\quad
\mbox{ as
  }\varepsilon \rightarrow 0.$$
${\bf b)}$ For any sequence of functions $w_\varepsilon \in
L^1((\Omega\setminus\overline{\Omega_1})\times]0,T[)$ and any $w\in
L^1(\Gamma\times]0,T[)$ such that
$$\|\nabla w_\varepsilon \|_{q,S_\varepsilon }\leq C\quad\mbox{ and
  }\quad\int^T_0\int_\Gamma (w_\varepsilon -w)dsdt\rightarrow 0,$$
for some constant $C>0$ and some exponent $q>1,$ we have
$$\int^T_0\int_{S_\varepsilon }{w_\varepsilon \over
  \varepsilon \gamma}dxdt\rightarrow\int_0^T\int_\Gamma wdsdt\qquad\mbox{ as
  }\varepsilon \rightarrow 0.$$\end{lemma}

For an arbitrary $v\in{\mathcal X}_\Gamma\hookrightarrow 
{\mathcal X}_\varepsilon\cap C([0,T];H^1(\Omega\setminus \overline{\Omega_1})) $, by
Lemma \ref{ljf} ${\bf a)}$ we have 
$$\int_{S_\varepsilon }{1\over
  2\varepsilon \gamma}|v(0)-u^0|^2dx\rightarrow\int_\Gamma{1\over  2}|v(0)-u^0|^2ds.$$

In order to apply Lemma \ref{ljf} {\bf b)}, we define
$w_\varepsilon =(v-u_\varepsilon )\partial_t v$ and $w=(v-u)\partial_t v.$ By
(\ref{linf}) we obtain
$$\int^T_0\int_\Gamma(w_\varepsilon -w)dsdt\rightarrow 0.$$

Since $\partial_t\nabla v\in L^2(\Omega\times]0,T[)$ we have
$$\|\nabla
w_\varepsilon \|_{q,S_\varepsilon }\leq\|\nabla(v-u_\varepsilon )\|_{2,S_\varepsilon }\|\partial_t
  v\|_{{2q\over 2-q},S_\varepsilon }+\|v-u_\varepsilon \|_{{2q\over
      2-q},S_\varepsilon }\|\partial_t\nabla v\|_{2,S_\varepsilon }$$
for $q>1$ satisfying $2q/(2-q)\leq 2n/(n-2)$
that means $q\leq n/(n-1).$

 Thus we can pass to the limit on $\varepsilon\rightarrow 0$ 
in (\ref{wteo4})
to obtain the desired solution.

\section{Proof of Theorem \ref{teo1}}
\label{s3}

The generalized version of the Poincar\'e inequality applied to functions admitting jumps
\cite{amar2005} can once more extended to the following version.
\begin{lemma}\label{lpoin}
Let ${\bf v}\in {\bf V}$.
Then
\begin{equation}\label{poinc}
\int_{\Omega_1} v_1^2dx\leq C\left\{
\int_\Omega |\nabla {\bf v}|^2dx+\int_\Gamma [v]^2ds\right\}.
\end{equation}
\end{lemma}
\begin{proof}
If $\Gamma_1\not=\emptyset$, the classical Poincar\'e inequality is valid and then (\ref{poinc})
clearly holds.
If $\Gamma_1=\emptyset$ we will prove by contradiction. Assuming that (\ref{poinc}) is not true, there exists a sequence $\{ {\bf v}_m\}\subset
 {\bf V}$ such that for all $m\in \mathbb N$
$$\|  v_{1m}\|_{2,\Omega_1}=1\quad
\mbox{and }\quad\|
\nabla  {\bf v}_m\|^2_{2,\Omega}+\|[v_m] \|^2_{2,\Gamma}\leq{1/m}.$$
Hence $\nabla {\bf v}_m\rightarrow {\bf 0}$ in ${\bf L}^2(\Omega)$ and 
$[v_m] \rightarrow 0$ in $L^2(\Gamma).$ Since
 $ {\bf V}$ is a reflexive Banach space, we can extract a subsequence of $ {\bf v}_m$,
 still denoted by $ {\bf v}_m$, such that $ {\bf v}_m\rightharpoonup  {\bf v}$ in 
${\bf V}.$ Thus 
$\nabla  {\bf v}={\bf 0}$ in $\Omega$ and 
$v_1=v_2$ on $\Gamma.$ Consequently 
 $v_1\in H^1_{\Gamma_1}(\Omega_1)$ and  $v_2\in H^1_{\Gamma_2}(\Omega_2)$
satisfy $v_1\equiv v_2\equiv 0$.
 From the compact embedding ${\bf V}\hookrightarrow\hookrightarrow
 L^2(\Omega_1)\times L^2(\Omega_2)$ it follows that
$${\bf v}_m\rightarrow {\bf 0}\quad\mbox{ in } L^2(\Omega_1)\times L^2(\Omega_2).$$

Then we conclude that
$$\|v_{1m}\|_{2,\Omega_1}=1\rightarrow\|0\|_{2,\Omega_1}=1$$
which is a contradiction. 
\end{proof}

\subsection{Discretization in time}
\label{time3}

As in Section \ref{time}, we will construct weak solutions
${\bf u}^{i+1}={\bf u}(t_{i+1,m})$, $i\in\{0,1,\cdot\cdot\cdot,m-1\},$ of an approximate
time-discrete problem.
\begin{prop}\label{propui}
 Let the assumptions (\ref{defs})-(\ref{defj})
 be valid, $m\geq\sigma_\#T/\alpha_\#$ and $i\in \{0,1,\cdot\cdot\cdot,m-1\}$ be fixed,
 $[u^i]\in L^2(\Gamma)$,
\[{\bf f}^{i+1}=
{\bf f}(t_{i+1,m})\in {\bf V}'\quad\mbox{ and }\quad
g^{i+1}=
g(t_{i+1,m})\in Y'.\]
Then there exists a time-discrete solution ${\bf u}^{i+1}\in {\bf V}$  to the problem
\begin{eqnarray}
\int_\Omega\sigma\nabla {\bf u}^{i+1}\cdot\nabla({\bf v}-{\bf u}^{i+1})dx
+\int_\Gamma{\alpha\over  h}[u^{i+1}]\left([v]-[u^{i+1}]\right)ds+\nonumber\\
+\langle g^{i+1},v_1-u^{i+1}_1\rangle_\Gamma+\int_\Gamma\{j([v])-j([u^{i+1}])\}ds\geq\langle
{\bf f}^{i+1},{\bf v}-{\bf u}^{i+1}\rangle_\Omega+\nonumber
\\+\int_\Gamma{\alpha\over h}[u^i]\left(
[v]-[u^{i+1}]\right)ds,\quad \forall{\bf  v}\in {\bf V},\label{ui}\end{eqnarray}
with $[u^0]=S$ on $\Gamma.$
\end{prop}
\begin{proof}
We show the existence of a solution to (\ref{ui}) with the aid of
the general theory on maximal monotone mappings applied to
elliptic variational inequalities \cite[pp. 874-875, 892-893]{z}.
To this end, we define the mapping $A:{\bf V}\rightarrow {\bf V}'$ by
$$\langle A{\bf u},{\bf  v}\rangle=\int_\Omega \sigma\nabla {\bf u}\cdot\nabla {\bf v}dx
+\int_\Gamma {\alpha\over h }{[u ]}[v]ds$$
which is single-valued, linear and hemicontinuous; and
the mapping $\varphi:{\bf V}\rightarrow [0,+\infty]$ by
$$\varphi({\bf v})=\left\{\begin{array}{ll}
\int_\Gamma j([v])ds,&\mbox{ if }
  j([v])\in L^1(\Gamma)\\
+\infty,&\mbox{ otherwise}
\end{array}\right.$$
which is convex, lower semicontinuous and $\varphi\not\equiv +\infty$. Because of 
(\ref{defs})-(\ref{defj})
 the coercivity condition
$$\langle A{\bf u},{\bf u}\rangle+\varphi({\bf u})
=\int_\Omega \sigma|\nabla {\bf u}|^2dx+\int_\Gamma {\alpha\over h}[u ]^2ds
+\int_\Gamma j([u])ds\geq  \sigma_\#\|{\bf u}\|_{\bf V}^2,$$
 is valid for any $h\leq\alpha_\#/\sigma_\#$.
Then, for ${\bf b}\in {\bf V}'$ such that
$$\langle {\bf b},{\bf  v}\rangle=-\langle {\bf f}^{i+1}, {\bf v}\rangle_\Omega+\langle g^{i+1}, v_1\rangle_\Gamma
-\int_\Gamma {\alpha\over h}[u^i][v]ds,$$
the variational inequality (\ref{ui}) has a unique weak solution
${\bf u}={\bf u}^{i+1}\in {\bf V}.$
\end{proof}
\begin{remark}
Since $[u^0]=S$ on $\Gamma$ means that $[u^0]\in L^2(\Gamma)$, then Proposition \ref{propui}
guarantees the existence of ${\bf u}^1\in {\bf V}$ and consequently $[u^1]\in L^2(\Gamma)$.
Therefore,  Proposition \ref{propui} successively
guarantees the existence of ${\bf u}^{i+1}\in {\bf V}$ for every $i=1,\cdots,m-1$.
\end{remark}

\subsection{Existence of a limit ${\bf u}$}

\begin{prop}\label{tilde}
 Let  $m\geq\sigma_\#T/\alpha_\#$.
For all  $i\in \{0,1,\cdots,m-1\}$, the estimate holds 
 \begin{equation}\label{dui1}
\alpha_\#\|[u^{i+1}]\|^2_{2,\Gamma}\leq
C(\|{\bf f}\|^2_{L^2(0,T;{\bf V}')}+\|g\|^2_{L^2(0,T;Y')}+ 
\|S\|_{2,\Gamma}^2).\end{equation}
Moreover, if
$\{\widetilde {\bf u}_m\}_{m\in\mathbb N}$ is the sequence defined by the step functions
$\widetilde {\bf u}_m:I\rightarrow {\bf V}$
$$\widetilde {\bf u}_m (t)=\left\{\begin{array}{ll}
{\bf u}^{1}&\mbox{ for }t=0\\
{\bf  u}^{i+1}&\mbox{ in }]t_{i,m},t_{i+1,m}]
\end{array}\right.$$
then there exists ${\bf u}$ such that
$$\widetilde {\bf u}_m\rightharpoonup {\bf u}\mbox{ in }L^2(0,T;{\bf V}).$$
\end{prop}
\begin{proof}
Testing in (\ref{ui}) with ${\bf v}={\bf 0}$ and using (\ref{defj}), we get
$$\int_\Omega\sigma|\nabla {\bf u}^{i+1}|^2dx+\int_\Gamma{\alpha\over  h}[u^{i+1}]^2ds\leq\langle 
{\bf f}^{i+1},{\bf u}^{i+1}\rangle_\Omega-\langle g^{i+1},
 u^{i+1}_1\rangle_\Gamma+\int_\Gamma {\alpha\over h}[u^i][u^{i+1}]ds,$$
 for all  $i\in \{0,1,\cdot\cdot\cdot,m-1\}$.
Hence, applying (\ref{defs})
and Lemma \ref{lpoin} it follows
\begin{eqnarray*}
{\sigma_\#\over 2} \| \nabla {\bf u}^{i+1}\|^2_{2,\Omega}+
\int_\Gamma{\alpha\over 2 h}[u^{i+1}]^2ds
\leq
{1\over 2\sigma_\#}\left(\|{\bf f}^{i+1}\|_{{\bf V}'}+C_Y\|g^{i+1}\|_{
Y'}\right)^2+\\
+{\sigma_\#\over 2} \|  [{u}^{i+1}]\|^2_{2,\Gamma}
+\int_\Gamma{\alpha\over  2 h}[u^i]^2ds,
\end{eqnarray*}
with $C_Y$ standing for the continuity constant of
$H^1_{\Gamma_1}(\Omega_1)\hookrightarrow Y$.
Summing on $k=0,...,i$, multiplying by $2h$  and applying (\ref{defa}), we find 
\begin{eqnarray*}
 \sigma_\#h\sum_{k=0}^i  \|{\bf u}^{k+1}\|^2_{\bf V}+ \alpha_\#  \| [u^{i+1}]\|^2_{2,\Gamma}
\leq {2\over \sigma_\#}h\sum_{k=1}^{i+1}(\|{\bf f}^k
\|^2_{{\bf V}'}+C_Y^2\|g^k\|^2_{Y'})+\\
+\sigma_\#h\sum_{k=0}^i  \|  [{u}^{k+1}]\|^2_{2,\Gamma}
+{\alpha^\#} \|S\|_{2,\Gamma}^2.
\end{eqnarray*}

Consequently, by the Gronwall Lemma we get (\ref{dui1}) and, for
$i=m-1$,
 \begin{equation}\label{ui1}
\|\widetilde{\bf  u}_m\|^2_{L^2(0,T;{\bf V})}\leq
C(\|{\bf f}\|^2_{L^2(0,T;{\bf V}')}+\|g\|^2_{L^2(0,T;Y')}+ 
\|S\|_{2,\Gamma}^2).\end{equation}
Thus we can extract a subsequence, still denoted by $\widetilde
{\bf u}_m,$ weakly convergent to ${\bf u}\in L^2(0,T;{\bf V}).$
\end{proof}

\begin{prop}\label{UU}
 Let  $m\geq\sigma_\#T/\alpha_\#$ and
$U_m:[0,T[\rightarrow L^2(\Gamma)$ be defined by
$$U_m (t)=\left\{\begin{array}{ll}
\displaystyle{[u^{1}]-S
\over h}&\mbox{ for }t=0\\&\\
\displaystyle{[u^{i+1}]-[u^{i}]
\over h}&\mbox{ in }]t_{i,m},t_{i+1,m}]
\end{array}\right.\mbox{ on }\Gamma.$$
If the assumptions (\ref{defs})-(\ref{defj}), (\ref{js2}) and
(\ref{deff})-(\ref{duo}) are fulfilled,
then the estimate holds
\begin{equation}\label{ui2}
\|\widetilde{\bf  u}_m\|^2_{L^\infty(0,T;{\bf V})}+\|U_m\|^2_{2,\Sigma}\leq
C(\|{\bf f}\|^2_{L^2(0,T;{\bf V}')}+\|g\|^2_{L^2(0,T;Y')}+\|{\bf u}^0\|_{\bf
V}^2).
\end{equation}
Hence, we can extract a subsequence, still denoted by $U_m,$ weakly
convergent to $U\in L^2(\Sigma).$
\end{prop}
\begin{proof}
 For a fixed $t$, there exists
$i\in\{0,\cdot\cdot\cdot,m-1\}$ such that  $t\in]t_{i,m};t_{i+1,m}]$.
Choosing ${\bf v}={\bf u}^i$ as a test function in (\ref{ui}), we have
\begin{eqnarray*}
\int_\Omega\sigma\nabla {\bf u}^{i+1}\cdot\nabla({\bf u}^{i+1}-{\bf u}^i)dx+
\int_\Gamma{\alpha\over  h}([u^{i+1}]-[u^i])^2ds+\int_\Gamma
 j([u^{i+1}])ds\leq\\ \leq\langle
g^{i+1},u_1^{i}-u^{i+1}_1\rangle_\Gamma+\int_\Gamma j([u^i])ds+\langle{\bf  f}^{i+1},
{\bf u}^{i+1}-{\bf u}^i\rangle_\Omega.
\end{eqnarray*}
Summing on $k=0,...,i$ and remarking that
\begin{eqnarray*}
\sum_{k=0}^i\int_\Omega\sigma\nabla{\bf  u}^{k+1}\cdot\nabla({\bf u}^{k+1}-{\bf u}^k)dx&=&
{1\over 2}\int_\Omega\sigma|\nabla {\bf u}^{i+1}|^2dx-{1\over 2}\int_\Omega\sigma
|\nabla {\bf u}^0|^2dx+\\&+&{1\over 2}\sum_{k=0}^i
\int_\Omega\sigma|\nabla({\bf u}^{k+1}-{\bf u}^k)|^2dx
\end{eqnarray*}
then we find
\begin{eqnarray}
{\sigma_\# \over 2}\|\nabla {\bf u}^{i+1}\|^2_{2,\Omega}+
\alpha_\#\sum_{k=0}^i h\int_\Gamma\left({[u^{k+1}]-[u^k]\over
h}\right)^2ds\leq {\sigma^\#
\over 2}\|\nabla {\bf u}^0\|_{2,\Omega}^2+\nonumber\\+
\int_\Gamma j(S)ds+\langle g^1,
u^0_1\rangle_\Gamma+\sum_{k=1}^i\langle g^{k+1}-g^k,u^k_1\rangle_\Gamma-\langle
g^{i+1},u^{i+1}_1\rangle_\Gamma\nonumber\\
-\langle {\bf f}^1,
{\bf u}^0\rangle_\Omega-\sum_{k=1}^i\langle {\bf f}^{k+1}-{\bf f}^k,{\bf u}^k\rangle_\Omega+
\langle{\bf f}^{i+1},{\bf u}^{i+1}\rangle_\Omega.\label{dtui}
\end{eqnarray}

 Using (\ref{deff})-(\ref{defg}),  it follows
 \begin{eqnarray*}
 \sum_{k=1}^i\langle{\bf  f}^{k+1}-{\bf f}^k,{\bf u}^k\rangle_\Omega\leq d_1 h
 \sum_{k=1}^i\|{\bf u}^k\|_{\bf V}; \\
 \sum_{k=1}^i\langle g^{k+1}-g^k,u^k_1\rangle_\Gamma\leq d_2hC_Y
 \sum_{k=1}^i\|{\bf u}^k\|_{\bf V}.
 \end{eqnarray*}
Therefore, inserting the above inequalities in (\ref{dtui}), applying (\ref{ui1}) and
gathering (\ref{dui1}), it results
(\ref{ui2}).
\end{proof}

  We again have to relate the weak limits $\bf u$ and $U$.
\begin{prop}\label{duu}
Let $\bf u$ and $U$ be the weak limits obtained in Propositions \ref{tilde} and \ref{UU}, 
respectively. Then
$$\partial_t[u]=U\mbox{ in }L^2(\Gamma),\mbox{ for almost all }t\in
I.$$
\end{prop}
\begin{proof}  For a fixed $t$, there exists
$i\in\{0,\cdot\cdot\cdot,m-1\}$ such that  $t\in]t_{i,m};t_{i+1,m}]$.
By construction
$$\int^t_0 U_m(\tau)d\tau=\sum_{k=0}^{i-1}\int_{kh}^{(k+1)h}
{[u^{k+1}]-[u^k]\over h}d\tau+\int_{ih}^{t}
{[u^{i+1}]-[u^i]\over h}d\tau\quad\mbox{ on }\Gamma.$$

Setting  the Rothe sequence $\{{\bf u}_m\}_{m\in\mathbb N}$ defined by
$${\bf u}_m(x,t)={\bf u}^{i}(x)+(t-t_{i,m}){{\bf u}^{i+1}(x)-{\bf u}^{i}(x)\over
  h{}}\mbox{ in } I_{i,m},$$
for all  $i\in \{0,1,\cdots,m-1\}$
(compare to Definition \ref{defro}) under 
   $m\geq\sigma_\#T/\alpha_\#$, it results
$$\int^t_0 U_m(\tau)d\tau=[u_m](t)-S\mbox{ on }\Gamma.$$
From the Riesz theorem we get
$$([u_m](t)-S
,v)=\int^t_0(U_m(\tau),v)d\tau,\quad\forall v\in L^2(\Gamma).$$
Indeed, the right hand side of the above equation is
a bounded linear functional in $L^2(\Gamma)$, representable thus
 (uniquely) by the element $[u_m](t)-S$ from $L^2(\Gamma).$
Also there exists $w\in C([0,T];{L^2(\Gamma))}$ such that
$$(w(t),v)=\int^t_0(U(\tau),v)d\tau,\qquad\forall v\in L^2(\Gamma).$$

Then we have
\[
\lim_{m\rightarrow+\infty}\left([u_m](t)-S-w(t),v\right)=
\lim_{m\rightarrow+\infty}\int^t_0(U_m(\tau)-U(\tau),v)d\tau=0.\]

Let us prove that
 the norms of the  functions $[u_m]$ are uniformly bounded with
respect to $t\in I$ and $m$. From  the estimates
(\ref{dui1}) independent on $i$ and $m$, and
considering 
$$\|[u_m](t)\|_{2,\Gamma }=\|[u^{i}]\left(1+{t-t_{i,m}\over
  h{}}\right)+[u^{i+1}]{t-t_{i,m}\over h{}}\|_{2,\Gamma}$$
then, we get
$$\|[u_m ]\|^2_{L^\infty(0,T;L^2(\Gamma))}\leq
C(\|{\bf f}\|^2_{L^2(0,T;{\bf V}')}+\|g\|^2_{L^2(0,T;Y')}+\|S\|^2_{2,\Gamma}).$$

Hence,
 the Lebesgue Dominated Convergence Theorem yields
$$\lim_{m\rightarrow+\infty} \int^T_0([u_m](t)-S-w(t),v)dt=0,\qquad\forall v\in L^2(\Gamma).$$
Proceeding as in the proof of Proposition \ref{duuw}, we end up with
$$[u](t)-S=\int^t_0 U(\tau)d\tau.$$
\end{proof}

\subsection{Passage to the limit on $m\rightarrow +\infty$}
\label{passm}

Denoting ${\bf f}_m(t)={\bf f}^{i+1}$ and $g_m(t)=g^{i+1}$ for
$t\in ]t_{i,m},t_{i+1,m}]$ and $i\in\{0,\cdot\cdot\cdot,m-1\}$, we have
\begin{eqnarray*}
\int_Q\sigma\nabla\widetilde  {\bf u}_m\cdot\nabla vdxdt+\int_0^T\langle
g_m,v_1- \tilde u_{m1}\rangle_\Gamma dt+
\int_\Sigma\alpha U_m[v]dsdt+\int_\Sigma j([v])dsdt\geq\\
\geq
\int_Q\sigma|\nabla\widetilde {\bf  u}_m|^2dxdt
+\int_\Sigma\alpha U_m[ \tilde
u_m]dsdt+\int_\Sigma j([\tilde  u_m])dsdt+\int_0^T\langle{\bf  f}_m,
{\bf v}-\widetilde {\bf  u}_m\rangle_\Omega dt.
\end{eqnarray*}

From Propositions \ref{tilde} and \ref{UU}
to pass to the limit the above inequality
and recalling the weak lower s.c. property for the first 
term on the right hand side of the above inequality,
it remains to prove that
$$[\tilde u_m]\rightarrow [u]\mbox{ in }L^2(\Sigma).$$

Taking $\widetilde {\bf u}_m-{\bf u}=\widetilde{\bf  u}_m-{\bf u}_m+{\bf u}_m-{\bf u}$
first let us prove that
$$[\tilde u_m]-[u_m]\rightarrow 0\mbox{ in }L^2(\Sigma).$$
Since we have $0<t-t_{i,m}\leq h\mbox{ in }]t_{i,m};t_{i+1,m}]$ we
obtain
$$\|[\tilde
u_m](t)-[u_m](t)\|_{2,\Gamma}=\|U_m\|_{2,\Gamma}(h-(t-t_{i,m}))<
h\|U_m\|_{2,\Gamma}.$$
Using (\ref{ui2}) we derive
$$\|[\tilde u_m]-[u_m]\|_{2,\Sigma}\leq{CT\over
  {m}}(\|{\bf f}\|^2_{L^2(0,T;{\bf V}')}+\|g\|^2_{L^2(0,T;Y')}+\|{\bf u}
^0\|_{\bf V}^2)^{1/2}\rightarrow
  0.$$

Secondly the Rothe sequence $\{{\bf u}_m\}$ is bounded in $L^2(0,T;{\bf V})$,
and, from Prop. \ref{duu},
the functions $\partial_t[u_m]$ are bounded in
$L^2(\Sigma)$ then, for a
subsequence still denoted by $[u_m],$ the strong convergence holds
$$[u_m]\rightarrow[u]\mbox{ in } L^2(\Sigma).$$
Then it results
$$\int^T_0\int_\Gamma U_m[\tilde
u_m]dsdt\rightarrow\int^T_0\int_\Gamma U[u]dsdt=
\int^T_0\int_\Gamma[\partial_t
u][u]dsdt.$$

Therefore we can pass to the limit 
to obtain the weak formulation (\ref{pbu}). 
From the standard technique to prove uniqueness of solution, the
solution $\bf u$ to (\ref{pbu}) with (\ref{cf})
 is unique. Then the whole sequence
$\{\widetilde {\bf u}_m\}$ converges weakly to ${\bf u}\in L^2(0,T;{\bf V}).$

\section{Regularity in time}
\label{reg1}

{\sc Proof of Theorem \ref{prop1}.} 
 The proof follows the time discretization argument
as in Theorem  \ref{teo1},
considering the existence of the integral inequality (\ref{ui}).
Testing in (\ref{ui}) for the solutions ${\bf u}^{i+1}$ and ${\bf u}^i$ with
${\bf v}=({\bf u}^{i+1}+{\bf u}^i)/2$, summing the consecutive integral inequalities,
 and dividing by $h$,  we deduce
\begin{eqnarray*}
\int_\Omega h\sigma|\nabla {\bf Z}^{i+1}|^2dx+\int_\Gamma\alpha(U^{i+1}-U^i)U^{i+1}ds\leq
\langle {\bf f}^{i+1}-{\bf f}^i,{\bf Z}^{i+1}\rangle_\Omega+\nonumber\\
+\langle g^i-g^{i+1},Z^{i+1}_1\rangle_\Gamma
\end{eqnarray*}
with $U^{i+1}=([u^{i+1}]-[u^i])/h$ on $\Gamma$ and ${\bf Z}
^{i+1}=({\bf u}^{i+1}-{\bf u}^i)/h\in {\bf V}$, and 
taking into account the convexity of $j$.
Applying  the relation
$2(a-b)a=a^2+(a-b)^2-b^2$ to $a=U^{i+1}$ and $b=U^i,$
 and the  assumptions (\ref{deff})-(\ref{defg}),
it results
$$\int_\Omega h\sigma|\nabla{\bf  Z}^{i+1}|^2dx+\int_\Gamma\alpha(U^{i+1})^2ds\leq
\int_\Gamma\alpha(U^{i})^2ds+
(d_1+C_Yd_2) h\|{\bf Z}^{i+1}\|_{\bf V}.$$
Notice that the $\bf V$-norm can be no equivalent to a seminorm.
Thus summing on  $k=1,\cdots,i$ ($i\in\{1,\cdots,m-1\}$) we obtain
\begin{eqnarray}
{\sigma_\#\over 2}\sum^{i}_{k=1}h\|\nabla {\bf  Z}^{k+1}\|^2_{2,\Omega}
+\alpha_\#\|U^{i+1}\|^2_{2,\Gamma}\leq\alpha^\# \|U^1\|^2_{2,\Gamma}+\nonumber\\+T{
(d_1+C_Yd_2)^2\over 2\sigma_\#}+{\sigma_\#\over 2} \sum^{i}_{k=1}h\|{U}^{k+1}\|_{2,\Gamma}^2,\label{uz}
\end{eqnarray}
with $mh=T$.

Let us determine the estimate for the first term on the right hand
side of the above inequality.
Rewrite the integral identity (\ref{ui}) for $i=0$ in the form
\begin{eqnarray*}
\int_\Omega\sigma\nabla({\bf u}^1-{\bf u}^0)\cdot\nabla ({\bf v}-{\bf u}^1)dx+\int_\Omega\sigma\nabla
{\bf u}^0\cdot\nabla ({\bf v}-{\bf u}^1)dx+\\
+\int_\Gamma\alpha{[u^1]-S\over h}([v]-[u^1])ds
  +\int_\Gamma\{j([v])-j([u^1])\}ds
\geq\langle {\bf f}^1-{\bf f}(0),{\bf v}-{\bf u }^1\rangle_\Omega+\\
+\langle{\bf f}(0) ,{\bf v}-{\bf u}^1\rangle_\Omega
-\langle g^1-g(0), v_1 -u^1_1\rangle_\Gamma-
\langle g(0) , v_1 -u^1_1\rangle_\Gamma,\end{eqnarray*}
for all ${\bf v}\in{\bf  V}$, and in particular ${\bf v}={\bf u}^0$.
Thus, we apply the assumption (\ref{uoi}) with ${\bf v}={\bf u}^1$ and divide by $h$ we deduce
$$\int_\Omega{\sigma h}|\nabla {\bf Z}^1|^2dx+\int_\Gamma\alpha\left({[u^1]-S\over h}
\right)^2ds\leq
 \left(\|{\bf f}^1-{\bf f}(0)\|_{{\bf V}'}+\|g^1-g(0)\|_{Y'}\right)\|{\bf Z}^{1}\|_{\bf V}.$$
Then, using (\ref{defs})-(\ref{defa}),
(\ref{deff})-(\ref{defg}) and taking the Young
 inequality into account for the right hand side, we get
$$\alpha_\#
\|U^1\|^2_{2,\Gamma}\leq {
(d_1+C_Yd_2)^2h\over 2\sigma_\#}+{\sigma_\#\over 2} h\|{U}^{1}\|_{2,\Gamma}^2.$$
Considering $h<\alpha_\#\min\{1/\sigma_\#,1\}$ we insert the resulting estimate for $U^1$ into (\ref{uz})
concluding
\[{\sigma_\#\over 2}\sum^{i}_{k=1}h\|\nabla {\bf  Z}^{k+1}\|^2_{2,\Omega}
+\alpha_\#\|U^{i+1}\|^2_{2,\Gamma}\leq(\alpha^\#+T){
(d_1+C_Yd_2)^2\over \sigma_\#}+{\sigma_\#\over 2} \sum^{i}_{k=1}h\|{U}^{k+1}\|_{2,\Gamma}^2.\]
Hence, applying the Gronwall Lemma $U_m$ is uniformly estimated in $L^\infty(0;T;L^2(\Gamma))$
and  successively ${\bf Z}_m$ is uniformly estimated in $L^2(0;T;{\bf V})$.
Therefore the existence of a solution ${\bf u}\in C([0,T];{\bf V})$ in accordance to Theorem
\ref{prop1} can be done by similar arguments of passage to the limit
(cf. Section \ref{reg2}).

{\bf Acknowledgement.}
The author wishes to express her gratitude to J.F. Rodrigues for suggesting
the problem and some stimulating conversations.

\end{document}